\documentclass[12pt,reqno]{amsart}
\usepackage{amssymb,amsfonts,latexsym,mathrsfs}

\usepackage[a4paper%
]{geometry}

\usepackage{tikz}
\usetikzlibrary{intersections}
\usetikzlibrary{calc}
\usetikzlibrary{patterns}

\newcommand{\ahat}{\hat{a}}
\newcommand{\bhat}{\hat{b}}

\usepackage{todonotes}

\usepackage{enumitem}
\setenumerate[1]{label=\rm{(\alph*)}}
\setenumerate[2]{label=\rm{(\roman*)}}

\usepackage{graphicx}

\newtheorem{thm}{Theorem}[section]
\newtheorem{cor}[thm]{Corollary}

\newtheorem{lemma}[thm]{Lemma}

\newtheorem{prop}[thm]{Proposition}

\newtheorem{note}[thm]{Notations}

\theoremstyle{definition}
\newtheorem{defn}[thm]{Definition}

\newtheorem{example}[thm]{Example}

\newtheorem{rmk}[thm]{Remark}

\newcommand{\N}{\mathbb{N}}
\newcommand{\R}{\mathbb{R}}

\newcommand{\cN}{\mathcal{N}}

\newcommand{\bdy}{\partial}

\renewcommand{\geq}{\geqslant}
\renewcommand{\leq}{\leqslant}

\newcommand{\Qhat}{\hat{Q}}

\theoremstyle{remark}

\DeclareMathOperator{\dom}{dom}

\DeclareMathOperator{\Int}{Int}

\DeclareMathOperator{\conv}{conv}
\DeclareMathOperator{\PLC}{PLC}

\DeclareMathOperator{\acc}{Acc} 

\newcommand{\oball}[2]{B_{#1}(#2)} 

\renewcommand{\|}{\mid}
\newcommand{\<}{\langle}
\renewcommand{\>}{\rangle}

\newcommand{\df}[1]{\textit{\textbf{#1}}}

\renewcommand{\phi}{\varphi}

\newcommand{\ainf}{\hat{a}_\infty}
\newcommand{\binf}{\hat{b}_\infty}

\newcommand{\yinf}{\hat{y}_\infty}

\newcommand{\du}{\biguplus} 
\newcommand{\ainfty}{\hat{a}_\infty}
\newcommand{\binfty}{\hat{b}_\infty}

\title{Neighbourhoods in root systems of infinite Coxeter groups}

\author{Yuhan Cai}
\address{College of Engineering, Peking University, China}
\email{caiyuhuan@pku.edu.cn}
\author{Xiang Fu}
\address{Beijing International Center for Mathematical Research, Peking University, China}
\email{fuxiang@math.pku.edu.cn}
\author{Lawrence Reeves}
\address{School of Mathematics and Statistics, University of Melbourne, Australia}
\email{lreeves@unimelb.edu.au}

\begin{document}

\begin{abstract}
Let $W$ be a finitely generated infinite Coxeter group, with $\Phi$ and $\Pi$ being the 
corresponding root system and set of simple roots respectively. It has been observed in 
\cite{HLR11} that the projections of elements of $\Phi$ onto suitably chosen hyperplanes, 
called \emph{normalized roots},  are contained in the convex hull of $\Pi$ (which is a compact set), 
and hence the set of all normalized roots may exhibit interesting asymptotical behaviours. In this paper
we investigate the topology of the limit set of the normalized roots and demonstrate a natural system of neighbourhoods around each 
limit point arising from a non-affine infinite dihedral reflection subgroup of $W$. 
\end{abstract}

\maketitle

%

\section{Introduction}
Coxeter groups are abstract groups generated by involutions subject only to a set of braid relations. Such groups are important mathematical objects arising in a multitudes of areas in mathematics and physics. There has been a rich theory for finite Coxeter groups, but comparatively, the theory of infinite Coxeter groups is not as developed. Existing classification
of general infinite Coxeter groups only covers a few infinite families of Coxeter groups. Given the abstract nature of 
a Coxeter group, one common approach to study them is to study the reflection group representations of Coxeter groups, where 
we study an isomorphic copy of an abstract Coxeter group which is generated by reflections with respect to certain hyperplanes in a real vector space. This approach affords us with a powerful tool, called the \emph{root system} (consisting of all representative normal vectors with respect to the reflecting hyperplanes and their conjugates). It is hoped that a deeper understanding of root systems for general infinite Coxeter groups may extend our knowledge of general infinite Coxeter groups. 

It is well-known that a Coxeter group is infinite if and only if its root system is infinite, and it is also well-known that the root system for any Coxeter group, finite or infinite, is discrete. It has recently been observed that the projections of elements in the root system of finitely generated infinite Coxeter groups onto suitably chosen hyperplanes may have non-trivial accumulation sets. Indeed,    
the study of such asymptotical behaviours associated to an infinite Coxeter group was initiated by Hohlweg et al in \cite{HLR11} and further developed in 
\cite{DHR13} and \cite{limit3}. 
We investigate the topology of the limit set of normalized roots and demonstrate a natural system of neighbourhoods around each limit point arising from an infinite dihedral reflection subgroup.


\section{Preliminaries}

The basic object we consider is a root system as specified by a \emph{Coxeter datum} 
 consisting of a real vector space equipped with a symmetric bilinear form and a collection of simple roots (see Definition \ref{def:datum}). The simple roots are not assumed to be linearly independent, but merely positively independent. 
 This is natural from the point of view of subsystems.
Similar definitions of `root basis' or `based root system' 
are contained in \cite{C06}, \cite{K},  \cite{HLR11} and \cite{DHR13}.
We recall the definition of the associated Coxeter group as well as the dominance partial order on the set of roots.

\begin{defn}
Let $V$ be a real vector space and $A\subset V$.
The \df{positive linear cone} of $A$ is
$$\PLC(A)=\{\Sigma_{i=1}^n \lambda_i a_i\| n\in \N, \lambda_i\geq0,\ \lambda_i\neq 0 \text{ for some $i$ } \}$$
A subset $A$ of $V$ is called \df{positively independent} if $0\notin\PLC(A)$.

\end{defn}

\begin{defn}
 \label{def:datum}
 A \df{Coxeter datum} is a triple $\mathscr{C}=(V,  \Pi, B)$, where
 $V$ is a real vector space, $B$ is a symmetric bilinear form on $V$ and  $\Pi$  is subset of $V$ 
 such that the following conditions are satisfied:
 \begin{enumerate}
 \item[(C1)] $\Pi$ is {positively independent};
\item[(C2)] 
\begin{enumerate}
\item for all $a\in \Pi$, $B(a, a)=1$ (and we define $m_{aa}=1$)
\item for distinct $a,b\in\Pi$, either
$B(a, b)=-\cos(\pi/m_{ab})$ for some integer $m_{ab}\geq 2$
or else $B(a, b) \leq -1$ (in which case we define $m_{ab}=\infty$).
\end{enumerate}
\end{enumerate}
The elements $m_{ab}\in\N\cap\{\infty\}$ are determined by the Coxeter datum.
The set $\Pi$ of the Coxeter datum is called a \df{root basis}.
\end{defn}

Notice that   (C2) implies that  $a\notin \PLC(\Pi\setminus\{a\})$ for all   $a\in \Pi$. Furthermore, (C1) together with (C2) yield that
$\{a, b\}$ is linearly independent for distinct $a, b\in \Pi$.

Given a Coxeter datum $\mathscr{C}=(V,\Pi,B)$,  let $(W, R)$ be the associated Coxeter system in the sense of \cite{NB} or \cite{H}.
That is, $R=\{\,r_a\mid a\in \Pi\,\}$ is a set in bijection with $\Pi$, and $W$ is the group defined by the following presentation:
$$\< R\| (r_ar_b)^{m_{ab}}=1 \quad\text{for all $a,b\in\Pi$ with $m_{ab}<\infty$} \>
$$
As with the usual geometric representation of a Coxeter group, 
there is a faithful action of $W$ on $V$ determined by setting
$$
r_a(v)=v-2B(v,a)a
$$
The action of 
$W$ on $V$   preserves $B$.
We call $(W, R)$ the \df{abstract Coxeter system} associated to the Coxeter datum $\mathscr{C}$. 

\begin{defn}
\label{def: root system}
Let $\mathscr{C}=(V,\Pi,B)$ be a Coxeter datum, and let $(W, R)$ be the associated abstract Coxeter system. 
The \df{root system} of $W$ in $V$ is the set 
$$\Phi_{\mathscr{C}}=\{\,wa \mid \text{$w\in W$ and $a\in \Pi$}\,\}\subset V$$
The set $\Phi^+_{\mathscr{C}}=\Phi_{\mathscr{C}}\cap \PLC(\Pi)$ is called the set of \df{positive roots} and the set 
$\Phi^-_{\mathscr{C}}=-\Phi^+_{\mathscr{C}}$ is called  the set of \df{negative roots}.
When the Coxeter datum is fixed, we will often write $\Phi$ rather than $\Phi_{\mathscr{C}}$.
\end{defn}

The following combines results from \cite{RB96}.
\begin{prop} 
\label{pp: anu3}
Let $\mathscr{C}=(V,\Pi,B)$ be a Coxeter datum, and let $(W, R)$ be the associated abstract Coxeter system.
\begin{enumerate}
\item Let $w\in W$ and $a\in \Pi$. Then 
 \begin{equation*}
\ell(wr_a) =
\begin{cases}
\ell(w)-1,  &\text{if}\quad wa\in \Phi^-\\
\ell(w)+1,  &\text{if}\quad wa\in \Phi^+
\end{cases}
\end{equation*}

\item $\Phi=\Phi^+\du\Phi^-_{\mathscr{C}}$, where $\du$ denotes disjoint union.

\item $W$ is finite if and only if $\Phi$ is finite.\qed
\end{enumerate}
\end{prop}

The elements of $T=\bigcup_{w\in W} w Rw^{-1}\subset G$ are called \df{reflections}. 
If $x\in\Phi$ then $x=wa$ for some $w\in W$ and $a\in\Phi$ and 
$
(wr_aw^{-1})x=-x
$.
For each $t\in T$ we let $\alpha_t$ denote the unique positive root $wa$ with the property that $wr_aw^{-1}=t$. It is  easily checked that 
the map $T\to \Phi^+$ given by $t \mapsto \alpha_t$ is a bijection.




We recall from \cite{BH93} a partial ordering on $\Phi$ which is central to the rest of this paper:
\begin{defn}
Given a Coxeter datum $\mathscr{C}=(V,\Pi,B)$, and associated $W$, for 
$x, y\in \Phi$, we say that $x$ \df{dominates} $y$ if 
$$ \{\,  w\in W\mid wx\in \Phi^-   \,\} \subseteq \{\,   w\in W\mid wy\in \Phi^-    \,\}$$
and write $x\dom y$.
\end{defn}

The following collects a number of elementary results on dominance:
\begin{lemma}[{\cite[Lemma 2.2]{BH93}}]
\label{lem:dom}
Given a Coxeter datum $\mathscr{C}=(V,\Pi,B)$, and associated $W$. Then

\begin{enumerate}
\item There is dominance between $x,y\in \Phi$ if and only if $B(x, y)\geq 1$.


\item Dominance is $W$-invariant. 
\end{enumerate}
\qed
\end{lemma}


Given a Coxeter datum $\mathscr{C}=(V,\Pi,B)$, let $V_{L.I.}$ be a
vector space over $\R$ with basis $\Pi_{L.I.}=\{\,e_a\mid a\in \Pi\,\}$ and let $B_e$ be the unique bilinear form on $E$
satisfying 
$$
B_e(e_a, e_b) =B(a, b) \text{, for all } a, b\in \Pi.
$$
Then $\mathscr{C}_{L.I.}=(\, V_{L.I.},\, \Pi_{L.I.}, \, B_e\,)$ is again a Coxeter datum.
Moreover, $\mathscr{C}_{L.I.}$ and $\mathscr{C}$ are
associated to the same abstract Coxeter system $(W, R)$. 

\begin{defn}
\label{df: free}
The Coxeter datum $\mathscr{C}_{L.I.}=(\, V_{L.I.},\, \Pi_{L.I.}, \, B_e\,)$ constructed 
as above is call the \df{free Coxeter datum} corresponding to $\mathscr{C}$.
\end{defn}

Standard arguments (as in \cite[Corollary 1.4]{Fu12})  yields that 
$$\phi_{\mathscr{C}_{L.I.}}\colon W\to G_{\mathscr{C}_{L.I.}}=
\langle\{\, \rho_{e_a}\mid a\in \Pi\,\}\rangle$$ is an isomorphism.
Furthermore, $W$ acts faithfully on on $V_{L.I.}$ via $r_a y =\rho_{e_a} y$ for all $a\in \Pi$ and
$y\in V_{L.I.}$.  

Let $f\colon V_{L.I.}\to V$ be the unique linear map satisfying $f(e_a)= a$, for all
$a\in \Pi$. It is readily checked that $B(f(x), f(y))=B_e(x, y)$, for all $x, y\in
V_{L.I.}$. 
Now  for all $a\in \Pi$ and $y\in V_{L.I.}$, 
\begin{align*}
 r_a(f(y))=\rho_a(f(y))=f(y)-2B(f(y), a)a &=f(y)-2B(f(y), f(e_a))f(e_a)\\
                                         &=f(y-2B(y, e_a)_e e_a)\\
                                         &=f(\rho_{e_a}y)\\
                                         &=f(r_a y).
\end{align*}
Then it follows that $wf(y)=f(wy)$, for all $w\in W$ and all $y\in V_{L.I.}$, since
$W$ is generated by $\{\,r_a\mid a\in \Pi\,\}$.
And a similar reasoning as that of Proposition~2.9 of \cite{HRT97} enables us to deduce: 

\begin{prop}\textup{\cite[Proposition 2.1]{Fu12}}
\label{pp: eqv}
The map $f\colon V_{L.I.}\to V$ defined in the preceding paragraph restricts to a $W$-equivariant bijection $\Phi_{\mathscr{C}_{L.I.}}\leftrightarrow \Phi_{\mathscr{C}}$.
\qed
\end{prop}

\section{Limit roots}

We recall the definition of the limit roots associated to a Coxeter Datum from \cite{HLR11}.
If the root system (and therefore the associated Coxeter group) is infinite we can consider the accumulation points of the image of $\Phi$ in the projective space $PV$. In this context, it is convenient to normalize the roots to lie on a fixed affine hyperplane in $V$.

\begin{defn}\textup{(\cite{HLR11}, \cite{DHR13})}
\label{def: trans}
Given a Coxeter datum $\mathscr{C}=(V,\Pi,B)$, an affine hyperplane $V_1$ of codimension $1$ in $V$ is 
called \df{transverse} (to $\Phi^+$) if for each simple root $a\in \Pi$ the ray $\R_{>0}a$  intersects $V_1$ in
exactly one point, and this unique intersection is denoted by $\widehat{a}_{V_1}$. Given a hyperplane $V_1$ transverse to 
$\Phi^+$, let $V_0$ be the hyperplane that is parallel to $V_1$ and contains the origin.
\end{defn}

\begin{rmk}
For a Coxeter datum $\mathscr{C}$, it follows from Proposition~\ref{pp: anu3}~(b) that it is 
always possible to find a hyperplane containing the origin that separates $\Phi^+$ and $\Phi^-$.
By suitably translating this hyperplane it is always possible to find a hyperplane transverse to  $\Phi^+$.
\end{rmk}

Let $V_1$ be a transverse hyperplane and let $V_0$ be as in the preceding definition. 
Let $V_0^{+}$ be the open hall space induced by $V_0$ that contains $V_1$. Observe that 
$V_0^{+}$ contains  $\PLC(\Pi)$. Since 
$\Phi_{\mathscr{C}}^+\subset \PLC(\Pi)\subset V_0^{+}$, and $V_1$ is parallel to the boundary of 
$V_0^{+}$, it follows that $\#(V_1\cap \R_{>0} \beta) =1$ for each $\beta\in \Phi^+_{\mathscr{C}}$. 
This leads to an alternative definition of transverse hyperplanes:

\begin{lemma}\textup{\cite{HLR11}}
\label{lem: trans}
Given a Coxeter datum $\mathscr{C}=(V,\Pi,B)$, an affine hyperplane $V_1$
is transverse if and only if $\#(V_1\cap \R_{>0} \beta) =1$ for each $\beta\in \Phi^+_{\mathscr{C}}$.
\qed
\end{lemma}

\begin{defn}
\label{df: norm}
Let $\mathscr{C}=(V,\Pi,B)$ be a Coxeter datum, 
let  $V_1$ be a transverse hyperplane  in $V$, and let $V_0$ be obtained from $V_1$ as in Definition~\ref{def: trans}. 
\begin{enumerate}

\item For each $v\in V\setminus V_0$,  the unique intersection point of $\R v$ 
and the transverse hyperplane $V_1$  is denoted $\widehat{v}$. The 
\df{normalization map} is $\pi_{V_1}\colon V\setminus V_0\to V_1$, $\pi_{V_1}(v)=\widehat{v}$.
Set $\widehat{\Phi}_{\mathscr{C}{V_1}}=\pi_{V_1}(\Phi_{\mathscr{C}})$, and the elements 
$\widehat{x}\in\widehat{\Phi}_{\mathscr{C}{V_1}}$ are called \df{normalized roots}.

\item For $\widehat{x}\in \widehat{\Phi}_{\mathscr{C}V_1}$, set $x^+$ to be the 
unique element in $\Phi_{\mathscr{C}}^+$ with $\pi_{V_1}(x^+)=\widehat{x}_{V_1}$. 

\item Let $\varphi_{V_1}\colon V\to \R$ be the unique linear map satisfying the requirement that
$\varphi_{V_1}(v)=0$ for all $v\in V_0$, and $\varphi_{V_1}(v)=1$ for all $v\in V_1$.
\end{enumerate}
\end{defn}

Observe that $\pi_{V_1}(-x)= \pi_{V_1}(x)$ for all $x\in V$, and  $\widehat{y} =\frac{y}{\varphi_{V_1}(y)}$ for all $y\in V\setminus V_0$. 

Also observe that $\widehat{\Phi}_{\mathscr{C}{V_1}}\subseteq\conv(\widehat{\Pi})$, where 
$\conv(X)$ denotes the convex hull of a set $X$, and $\widehat{\Pi}=\{\,\widehat{x}\mid x\in \Phi_{\mathscr{c}} \,\}$.
If $\Pi$ in the Coxeter datum $\mathscr{C}=(\,V, \Pi, B\,)$ is a finite set (in which case the associated Coxeter 
group $W$ is finitely generated), then we see that $\widehat{\Phi}_{\mathscr{C}{V_1}}$ is contained in the compact set
$\conv(\widehat{\Pi})$. Consequently, if $\Pi$ is finite, then the accumulation points of $\widehat{\Phi}_{\mathscr{C}V_1}$
are contained in $\conv(\widehat{\Pi})$.   

\begin{defn}
Keep all the notations of the previous definition.
\begin{enumerate}
\item The set of \df{limit roots}  $E(\Phi_{\mathscr{C}V_1})$ (with respect to $V_1$)
 is the set of accumulation points of 
$\widehat{\Phi}_{\mathscr{C}V_1}$. 

\item The \df{isotropic cone} $Q$ (associated to the bilinear form $B$) is the set
$$Q=\{\, v\in V\mid B(v, v)=0\,\}.$$ Furthermore, define the \df{normalized isotropic cone}  $\widehat{Q}_{V_1}$ (with respect to $V_1$) by
 $\widehat{Q}_{V_1}= Q\cap V_1$.
 \end{enumerate}
\end{defn}

\begin{example}
\label{eg: special}
Let $\mathscr{C}_{L.I.}=(\,V_{L.I.}, \Pi_{L.I.}, B_e\,)$ be the free Coxeter datum in Definition \ref{df: free}.  For each
$v\in V_{L.I.}$, there is a unique expression of the form $v=\sum_{a\in \Pi} v_a a$, $v_a\in \R$. Then the hyperplane
$V_1:=\{\,v\in V_{L.I.} \mid \sum_{a\in \Pi} v_a=1\,\}$ is transverse to $\Phi_{\mathscr{C}_{L.I.}}^+$, and
$V_0:=\{\, v\in V\mid \sum_{a\in \Pi} v_a=0    \,\}$ is the corresponding hyperplane obtained from translating $V_1$ 
to contain the origin. Observe that under these conditions the corresponding $\varphi_{V_1}$ has the property that 
$\varphi_{V_1}(v)=\sum_{a\in \Pi} v_a$ for all $v\in V_{L.I.}$. 
\end{example}


\begin{thm}\textup{\cite[Theorem 2.7]{HLR11}}
\label{thm: limit}
Let $\mathscr{C}=(V,\Pi,B)$ be a Coxeter datum, and let $V_1$ be a transverse hyperplane, and 
let $\widehat{Q}$ be the corresponding normalized isotropic cone. Then
\begin{enumerate}
\item $E(\Phi_{\mathscr{C} V_1})\neq \emptyset$ if and only if the  Coxeter
group $W$ is infinite.

\item $E(\Phi_{\mathscr{C}V_1})\subseteq \widehat{Q}$.   
\end{enumerate}
\qed
\end{thm}

%

%
%



\begin{prop}
\label{pp:compact}
Let $\mathscr{C}=(V, \Pi, B)$ be a Coxeter datum in which $|\Pi|$ is finite, and suppose that the corresponding Coxeter group
$W$ is infinite. Then the set of limit roots $E(\Phi_{\mathscr{C}V_1})$ is compact.
\end{prop}

\begin{proof}
It follows from Theorem~\ref{thm: limit} that 
$$E(\Phi_{\mathscr{C}V_1})\subseteq \widehat{Q}\cap \conv(\widehat{\Pi}).$$
Since $E(\Phi_{\mathscr{C}V_1})$, by definition, is topologically closed, and
since $\conv(\widehat{\Pi})$ is bounded whenever $|\Pi|$ is finite, it follows 
that $E(\Phi_{\mathscr{C}V_1})$ is compact. 
\end{proof}

Following the convention set in \cite{HLR11}, we define
$$D=\bigcap_{w\in W} w(V\setminus V_0)\cap V_1 = V_1\setminus \bigcup_{w\in W} w V_0,$$
and we define the $\cdot$ action of $W$ on $D$ as follows: for any $w\in W$ and $x\in D$,
$$w\cdot x =\widehat{wx}.$$
Observe that the property that $WD=\bigcup_{w\in W}w D \subseteq V\setminus V_0$ 
guarantees that this $\cdot$ action of $W$ on $D$ is well-defined. Furthermore, 
we observe that each $w\in W$ acts continuously on $D$. The next result is taken from \cite{HLR11}
which summarizes a number of key facts:

\begin{prop}\textup{\cite[Proposition 3.1]{HLR11}}
\label{pp:D}
\begin{itemize}
\item[(i)] $\widehat{\Phi}_{\mathscr{C}V_1}$ and $E(\Phi_{\mathscr{C}V_1})$ are contained in $D$.
\item[(ii)] $\widehat{\Phi}_{\mathscr{C}V_1}$ and $E(\Phi_{\mathscr{C}V_1})$ are stable by the $\cdot$ action of $W$; 
moreover $\widehat{\Phi}_{\mathscr{C}V_1}= W\cdot \Pi$.
\item[(iii)] The topological closure $\widehat{\Phi}_{\mathscr{C}V_1}\uplus E(\Phi_{\mathscr{C}V_1})$ is stable under 
the $\cdot$ action of $W$.
\end{itemize}
\qed
\end{prop}

Furthermore, it has been observed in \cite{HLR11} that the $\cdot$ action has the following nice geometric description:

\begin{prop}\textup{\cite[Proposition 3.5]{HLR11}}
\begin{itemize}
\item[(i)] Let $\alpha\in \Phi_{\mathscr{C}}$, and $x\in D\cap Q$. Denote by $L(\widehat{\alpha}, x)$ the line containing
           $\widehat{\alpha}$ and $x$. Then
					\begin{itemize}
					\item[(a)] if $B(\alpha, x)=0$, then $L(\widehat{\alpha}, x)$ intersects $Q$ only at $x$, and $r_{\alpha}\cdot x =x$;
					\item[(b)] if $B(\alpha, x)\neq 0$, then $L(\widehat{\alpha}, x)$ intersects $Q$ in two distinct points, namely, $x$ and 
					           $r_{\alpha}\cdot x$.
					\end{itemize}
\item[(ii)] Let $\alpha_1$ and $\alpha_2$ be two distinct roots in $\Phi_{\mathscr{C}}$, $x\in L(\widehat{\alpha_1}, \widehat{\alpha_2})\cap Q$, and $w\in W$. Then $w\cdot x\in L(w\cdot \alpha_1, w\cdot \alpha_2)\cap Q$.					
\end{itemize}
\qed
\label{pp:geom}

\end{prop}

\section{Topology on $E$}

Let $\mathscr{C}=(\, V, \Pi, (\,,\,)\,)$ be a Coxeter datum, and let 
$V_1$ be a transverse hyperplane.  
The aim of this section is to describe the topology on $E(\Phi_{\mathscr{C}V_1})$ (induced from that of $V$)
in terms of sets defined using the dominance partial order.

For simplicity, we assume that 
$\mathscr{C}$  \emph{is equal to its own free Coxeter datum} (as defined in Definition~\ref{df: free}).
For each $v\in V$, $a\in \Pi$ write $v_a$ for the coordinate for $v$ with respect to $a$, and furthermore, let $V_1$ be the affine hyperplane given by $V_1=\{\,v\in V\mid \sum_{a\in \Pi} v_a=1\,\}$. As we have observed in Example~\ref{eg: special} that $V_1$ is a transverse hyperplane for $\Phi_{\mathscr{C}}^+$. The corresponding hyperplane $V_0$ is given by 
$V_0=\{\, v\in V\mid \sum_{a\in \Pi} v_a =0\,\}$. Consequently, the corresponding map $\varphi_{V_1}\colon V\to \R$ is then given by the rule that $\varphi_{V_1}(v)=\sum_{a\in \Pi} v_a$ (as was implicit in Example~\ref{eg: special}).   
We shall study the topology on $E(\Phi_{\mathscr{C}V_1})$ for this specific Coxeter datum $\mathscr{C}$ and the specific transverse hyperplane $V_1$. By applying the $W$-equivariant map $f$ of Proposition~\ref{pp: eqv} and natural  
homeomorphisms between different transverse planes, the same topological results shall hold for more general Coxeter data and transverse hyperplanes.  

To simplify notation, we shall adopt the following:
\begin{note}
\noindent\rm{(i)}\quad Write $\Phi$ for $\Phi_{\mathscr{C}}$, and 
$\Phi^+$ (respectively, $\Phi^-$) for $\Phi_{\mathscr{C}}^+$ (respectively, $\Phi_{\mathscr{C}}^-$).

\noindent\rm{(ii)}\quad Remove the subscript $V_1$ in $\widehat{\cdot}_{V_1}$ and $\varphi_{V_1}$.

\noindent\rm{(iii)}\quad For $v\in V$, write $|v|$ for $\varphi_{V_1}(v)$ (this is denoted as $|v|_1$ in \cite{HLR11}). 

\noindent\rm{(iv)}\quad If $W'$ is a reflection subgroup of $W$, 
then write $E(\Phi_{W'})$ for the set of accumulation points for $\widehat{\Phi_{W'}}$.

\noindent\rm{(v)}\quad Write $E$ in place of $E(\Phi_{ \mathscr{C} V_1})$. 

\noindent\rm{(v)}\quad Write $Q$ for the isotropic cone of the bilinear form $B$. 

\end{note}

Throughout this section, let $a,b\in\Phi^+$ be such that $B(a,b)<-1$.
We may set $\cosh \theta =-B(a, b)$ for some $\theta > 0$. 
Let $W_{a,b}$ be the dihedral reflection subgroup of $W$ generated by $r_a$ and $r_b$.
Set $V'=\R a\oplus \R b$, $\Pi'=\{a, b\}$, and let $B'(\,,\,)$ be the restriction of the bilinear form $B(\,,\,)$ on $V'$.
Then $\mathscr{C}'=(V', \Pi', B')$ is a Coxeter datum. It can be observed that $\Phi^+_{\mathscr{C}'}\subset \Phi^+$.
Set $E(W_{a,b})=E(\Phi_{\mathscr{C}'})$.
It follows from direct calculations that $(\R a\oplus \R b)\cap Q$ consists of two lines, given by: 

$$(\R a \oplus \R b)\cap Q = \R((\cosh\theta +\sinh\theta)a+b)\cup \R((\cosh\theta-\sinh\theta)a+b).$$
Thus $E(W_{a, b})\subset E$ is the intersection 
$$((\R a\oplus \R b)\cap Q)\cap V_1=\widehat{(\R a\oplus \R b)\cap Q}.$$
More explicitly, 
\begin{align*}
E(W_{a, b}) =\bigg\{\,&\frac{(\cosh\theta+\sinh\theta)|a|}{(\cosh\theta +\sinh\theta)|a|+|b|}\widehat{a}+
                 \frac{|b|}{(\cosh\theta +\sinh\theta)|a|+|b|}\widehat{b},\\ 
								&\frac{(\cosh\theta-\sinh\theta)|a|}{(\cosh\theta -\sinh\theta)|a|+|b|}\widehat{a}+
                 \frac{|b|}{(\cosh\theta -\sinh\theta)|a|+|b|}\widehat{b}	
									\,\bigg\}.
\end{align*}									
For each $i\in \N$, we adopt the following notation 
\begin{equation}
\label{eq:ci}
c_i:= \frac{\sinh(i\theta)}{\sinh \theta}.
\end{equation}
 Then
\begin{align*}
(r_a r_b)^i a =c_{2i+1} a + c_{2i} b,\\
\noalign{\hbox{and}}
(r_b r_a)^i b =c_{2i} a + c_{2i+1} b. 
\end{align*}
Observe that 
\begin{align}
\label{eq:grad}
\lim_{i\to \infty}\frac{c_{2i+1}}{c_{2i}}
&=\lim_{i\to \infty} \frac{\sinh (2i\theta)\cosh \theta+\cosh(2i\theta)\sinh\theta}{\sinh(2i\theta)}\notag\\
&=\lim_{i\to \infty} (\cosh\theta +\coth(2i \theta)\sinh\theta)\notag\\
&=\cosh\theta +\sinh\theta.
\end{align}
Consequently, we see that 
\begin{align*}
\ainfty &= \lim_{i\to\infty}\widehat{(r_a r_b)^i a}\in \R((\cosh\theta+\sinh\theta)a+b)\\
\noalign{\hbox{and}}
\binfty &= \lim_{i\to \infty}\widehat{(r_b r_a)^i b} \in \R((\cosh\theta-\sinh\theta)a+b)
        = \R(a+(\cosh\theta+\sinh\theta)b).
\end{align*}
Thus 
\begin{align*}
\ainfty &=\frac{(\cosh\theta+\sinh\theta)|a|}{(\cosh\theta +\sinh\theta)|a|+|b|}\widehat{a}+
                 \frac{|b|}{(\cosh\theta +\sinh\theta)|a|+|b|}\widehat{b}\\
\noalign{\hbox{and}}
\binfty &=\frac{(\cosh\theta-\sinh\theta)|a|}{(\cosh\theta -\sinh\theta)|a|+|b|}\widehat{a}+
                 \frac{|b|}{(\cosh\theta -\sinh\theta)|a|+|b|}\widehat{b}.								
\end{align*}
And it follows readily that 
\begin{align*}
B(\ainfty, a) &=\frac{\sinh\theta}{(\cosh\theta +\sinh\theta)|a|+|b|} >0;\\
\noalign{\hbox{and}}
B(\ainfty, b) &=\frac{-\sinh\theta(\cosh\theta +\sinh\theta)}{(\cosh\theta +\sinh\theta)|a|+|b|} <0;\\
\noalign{\hbox{whereas}}
B(\binfty, a) &=\frac{-\sinh\theta}{(\cosh\theta -\sinh\theta)|a|+|b|}<0;\\
\noalign{\hbox{and}}
B(\binfty, b) &=\frac{\sinh\theta (\cosh\theta-\sinh\theta)}{(\cosh\theta -\sinh\theta)|a|+|b|}>0.
\end{align*}

For each $i\in \N$, defining $a_i=(r_ar_b)^{i}a$, and $b_i=(r_b r_a)^i b$. Note that
$$
a_{i+1}\dom a_i, \, B(a_i,\ainf)>0, \,   \text{and}\, \hat a_i\to \ainfty.
$$
For $i\in\N$ define
$$
A_i=\{  x\in\Phi^+\|  x\dom a_i\},
$$
and
$$
N_i=\acc\{ \hat x\| x\in A_i\}\subset E,
$$
where $\acc(X)$ of a set $X$ denotes the set of accumulation points of $X$.

It is clear from the definition that  $N_{i+1}\subset N_i$ and that $N_i$ is a closed subset of $V$ (and therefore $E$).

\begin{lemma}\label{lemma:neighbourhood}
$N_i$ is a neighborhood of $\ainf$. That is, there exists $\epsilon_i>0$ such that 
$B_{\epsilon_i}(\ainf)\cap E\subset N_i$ (where $B_{r}(x)$ is the open ball of radius $r$ centered at $x$).
\end{lemma}
\begin{proof}
Since $B(a_i,\ainf)>0$ there exists $\epsilon_i>0$ and $\delta_i>0$ such that
$B(a_i,x)>\delta_i$ for all $x\in\oball{\epsilon_i}{\ainf}$. Let $\eta\in\oball{\epsilon_i}{\ainf}\cap E$.
If $\{x_i\}\subset\Phi^+$ is a sequence with $\hat x_j\to\eta$, it is necessarily the case that
$B(\hat{x}_j,a_i)>\delta_i$ (for all large enough $j$). If $B(x_j,a_i)\in(0,1)$ for all but finitely many $j$, then
$B(\hat{x}_j,a_i)\to 0$. 
It follows that, after possibly dropping to an infinite subsequence, we have $B(x_j,a_i)\ge 1$ for all $j$.
Then apply lemma \ref{lem:dom}\ref{lem:dom}.
\end{proof}


The following is a technical result which may become useful later on.

\begin{prop}\label{lem:periodic}
Suppose that $c\in V\setminus W_{a, b} V_0$ is such that $B(a,c)$ and $B(b,c)$ are not both zero.
Then $\xi:=\widehat{\lim_{i\to \infty}(r_ar_b)^ic}$ exists. More specifically, 
\begin{itemize}
\item[(1)] If $r_a r_b c \in \R c$, then $\xi = \widehat{c}\in \{\,\ainfty, \binfty\,\}\subset E$.
In particular, if $c=\binfty$, then $\xi=\binfty$.
\item[(2)] If $r_a r_b c\notin \R c$, then $\xi =\ainfty$.
\end{itemize}
\end{prop}

\begin{proof}
\noindent\rm{(1)}:\quad If $r_a r_b c\in \R c$, then $r_b c =\lambda r_a c$ for some $\lambda \in R$, and it is easily checked that
$c\in \R a+ \R b$. Furthermore, a direct rank-$2$ calculation then shows that $\xi = \widehat{c}\in \{\,\ainfty, \binfty\,\}$.

\noindent\rm{(2)}: We may assume that $c\notin \R a\oplus \R b$. Observe that 
\begin{align*}
r_a r_b c = r_a (c-2 B(b , c) b)&=c-2 B(a, c)a -2B(b, c)(b-2 B(b, a)a)\\
                                &=c+(4B(b,c)\, B(a,b) - 2B(a, c)) a -B(b,c)b\\
																&=c+\mu_1 a +\mu_2 b
\end{align*}
where $\mu_1 = 4B(b, c)\,B(a, b)-2B(a, c)$, and $\mu_2 =-2B(b,c)$. Then it follows from an easy induction that for all $i\in \N$, 
$$(r_a r_b)^i c= c+\mu_1(a+r_a r_b a+ \cdots +(r_a r_b)^{i-1}a)+\mu_2(b+r_a r_b b +\cdots +(r_a r_b)^{i-1}b). $$
Note that $\{a, b\}$ is linearly independent. Setting $A$ to be the matrix representing the action of $r_a r_b$ on the 
subspace $\R a\oplus \R b$ with respect to the ordered basis $\{a, b\}$. Then
\begin{align*}
A&=\begin{bmatrix}
-1 & -2B(a, b)\\
0  & 1
\end{bmatrix}
\begin{bmatrix}
1         & 0 \\
-2B(a, b) & -1
\end{bmatrix}\\
&=\begin{bmatrix}
4\cosh^2\theta -1  & -2\cosh\theta\\
2\cosh\theta & -1
\end{bmatrix}.
\end{align*}
An easy induction then yields that for all $i\in \N$, 
$$A^i =\begin{bmatrix}
c_{2i+1} & -c_{2i}\\
c_{2i}   & -c_{2i-1}
\end{bmatrix}, $$
where $c_n$ is as defined as in Equation (\ref{eq:ci}). It is then clear that
\begin{align*}
&\quad I + A + A^2 +\cdots A^{i-1}=(I-A)^{-1}(I -A^i)\\
&=\frac{1}{4-4\cosh^2\theta}
\begin{bmatrix}
2 & -2\cosh\theta\\
2\cosh\theta & 2-4\cosh^2\theta
\end{bmatrix}
\begin{bmatrix}
1-c_{2i+1} & c_{2i}\\
-c_{2i}    & 1+c_{2i-1}
\end{bmatrix}\\
&=\frac{1}{4-4\cosh^2\theta}
\begin{bmatrix}
u_i & u'_i \\
v_i & v'_i
\end{bmatrix}
\end{align*}
where 
\begin{align*}
u_i&=2(1-c_{2i+1})+2\cosh\theta c_{2i},\qquad v_i&=2\cosh\theta (1-c_{2i+1})-(2-4\cosh^2\theta)c_{2i},\\ 
u'_i&=2c_{2i}-2\cosh\theta(1+c_{2i-1}), \text{ and } v'_i &=2\cosh\theta c_{2i}+(2-4\cosh^2\theta)(1+c_{2i-1}).
\end{align*}
Hence 
$$
(r_a r_b)^i c= c+\frac{\mu_1}{4-4\cosh^2 \theta}(u_i a +v_i b)+\frac{\mu_2}{4-4\cosh^2\theta}(u'_i a + v'_i b). 
$$
Observe that 
\begin{align*}
\lim_{i\to \infty}\frac{u_i}{v_i}&=\lim_{i\to\infty}\frac{1-c_{2i+1}+\cosh\theta c_{2i}}{\cosh\theta (1-c_{2i+1})-(1-2\cosh^2\theta) c_{2i}}\\
&=\lim_{i\to\infty}\frac{\frac{1-c_{2i+1}+\cosh\theta c_{2i}}{c_{2i}}}{\frac{\cosh\theta (1-c_{2i+1})-(1-2\cosh^2\theta) c_{2i}}{c_{2i}}} \\
&=\lim_{i\to\infty}\frac{-\cosh \theta -\sinh \theta+\cosh \theta}{-\cosh\theta (\cosh \theta +\sinh\theta)-1+2\cosh^2\theta}\\
&=\lim_{i\to \infty}\frac{1}{\cosh\theta-\sinh\theta}\\
&=\cosh\theta+\sinh\theta.
\end{align*}
Similarly, we note that 
\begin{align*}
\lim_{i\to\infty}\frac{u'_i}{v'_i} &=\lim_{i\to\infty}\frac{ c_{2i}-\cosh\theta(1+c_{2i-1})}{\cosh\theta c_{2i}+(1-2\cosh^2\theta)(1+c_{2i-1})}\\
&=\lim_{i\to\infty}\frac{\frac{c_{2i}-\cosh\theta(1+c_{2i-1})}{c_{2i-1}}}{\frac{\cosh\theta c_{2i}+(1-2\cosh^2\theta)(1+c_{2i-1})}{c_{2i-1}}}\\
&=\lim_{i\to\infty}\frac{\cosh\theta+\sinh\theta-\cosh\theta}{\cosh\theta(\cosh\theta+\sinh\theta)+1-2\cosh^2\theta}\\
&=\lim_{i\to\infty}\frac{\sinh\theta}{\sinh\theta (\cosh\theta-\sinh\theta)}\\
&=\cosh\theta+\sinh\theta.
\end{align*}
Thus, similar as in (\ref{eq:grad}), we see that $(r_a r_b)^i c$ tends to point in the direction of the line $\R((\cosh\theta+\sinh\theta)a+b)$ as $i\to \infty$.
Hence $\xi =\lim_{i\to\infty}\widehat{(r_a r_b)^i c}$ exists.
Since $\R((\cosh\theta+\sinh\theta)a+b)\cap V_1=\{\ainfty\}$, it follows that $\xi=\ainfty$.
\end{proof}

\begin{lemma}\label{lemma:interlaced}
Suppose that $\{y_j\}_{j\in \N}\subset\Phi^+$ is a sequence such that for all $j$ there exists $i_j$ with
$$
a_{i_j}\dom y_j\dom a_j
$$
Then $\{\hat{y_j}\}$ converges to $\ainf$.
\end{lemma}
\begin{proof}
For each $j$ define 
$$
k_j=\min\{ i \| (r_b r_a)^{i+1}y_j\in\Phi^- \}.
$$ 
Note that this minimum exists since
$(r_b r_a)^{i_j+1}a_{i_j}\in\Phi^-$ so 
$(r_b r_a)^{i_j+1}y_j\in\Phi^-$. Also, $(r_b r_a)^k y_j\in\Phi^+$ for all $k\leq j$ since
$(r_ar_b)^ka_j\in\Phi^+$ and $y_j\dom a_j$.
Thus $j\leq k_j \leq i_j$, and $ (r_b r_a)^{k_j} y_j\in \Phi^+$, and $(r_b r_a)^{k_j+1} y_j\in \Phi^-$.
Then we may conclude that $y_j=(r_ar_b)^{k_j}c$, where $c\in \Phi^+$ and $(r_b r_a) c\in \Phi^-$. 
Then the desired result follows from Proposition \ref{lem:periodic}. 
\end{proof}

Following \cite{DHR13}, we set 
\begin{align*}
\mathcal{K}&=\{\, v\in \PLC(\Pi)\mid B(v, a)\leq 0, \text{ for all $a\in \Pi$} \,\}\\
\noalign{\hbox{and}}
\mathcal{Z}&=\bigcup_{w\in W} w \mathcal{K}, 
\end{align*}
where $\mathcal{Z}$ is the \emph{imaginary cone} of $W$ in $V$, and $\overline{\mathcal{Z}}$ denotes the topological 
closure of $\mathcal{Z}$. It has been observed in (2.2) of \cite{DHR13} that for all $x, y\in \overline{\mathcal{Z}}$, 
\begin{equation}
\label{eq:-ve}
B(x, y)\leq 0.
\end{equation}
Furthermore, we set $Z =\mathcal{Z}\cap V_1$, and set $\overline{Z}$ to be the topological closure of $Z$.
It has also been established in Theorem 2.2 of \cite{DHR13} that 
\begin{equation}
\label{eq:conv}
\conv (E)=\overline{Z}.
\end{equation}
Combining (\ref{eq:-ve}) and (\ref{eq:conv}), we easily deduce the following. 

\begin{lemma}
\label{lem:-ve}
Suppose that $x, y\in E$. Then 
$$B(x, y)\leq 0.$$
\qed
\end{lemma}

The following result is taken from \cite{Fu13}:
\begin{prop}\textup{\cite[Proposition 3.6]{Fu13}}
\label{pp:max}
Suppose that $\alpha, \beta\in \Phi^+$ be distinct. Let 
$$W'=\langle \{\,r_{\gamma}\mid \gamma\in (\R\alpha \oplus \R\beta)\cap \Phi^+ \,\}\rangle.$$
Then $W'$ is a dihedral reflection subgroup of $W$, furthermore, it is the maximal dihedral reflection subgroup containing 
$\langle r_{\alpha}, r_{\beta}\rangle$.
\qed
\end{prop}

Let $\gamma_1, \gamma_2\in \Phi$ be such that $r_{\gamma_1}, r_{\gamma_2}\in W'$.
Then it is clear from the above proposition that $L(\widehat{\gamma_1}, \widehat{\gamma_2}) = L(\widehat{\alpha}, \widehat{\beta})$.

For $x, y\in \Phi^+$, it has been observed in Section~2.3 of \cite{HLR11} that the cardinality of 
$L(\widehat{x}, \widehat{y})\cap\widehat{Q}$ is either $0$ (when $|B(x, y)|<1$), $1$ (when $|B(x, y)|=1$), or
$2$ (when $|B(x, y)|>1$). 

\begin{prop}
\label{pp:sep}
\begin{itemize}
\item[(i)] Suppose that $x, y\in \Phi^+$ such that $L(\widehat{x}, \widehat{y}) \cap\widehat{Q} =\{x_{\infty}, y_{\infty}\}$.
Then on the line $L(\widehat{a}, \widehat{b})$, there is no normalized root in between the two intersection points $x_{\infty}$ and $y_{\infty}$. 
\item[(ii)] Suppose that $x\dom y \in \Phi^+$ with $x\neq y$. Then on the line
$L(\widehat{x}, \widehat{y})$, the normalized root $\widehat{x}$ separates $\widehat{y}$ and $\widehat{Q}$.
\end{itemize} 
\end{prop} 

\begin{proof}
\noindent\rm (i):\quad If for a contradiction that there is $x\in \Phi$ with $\widehat{x}$ being on the line
line $L(\widehat{a}, \widehat{b})$ in between of $\widehat{a}_\infty$ and $\widehat{b}_\infty$. Then it follows from 
(\ref{eq:conv}) that $\widehat{x}\in \overline{Z}$, and consequently by (\ref{eq:-ve}) we must have
$B(\widehat{x}, \widehat{x})\leq 0$, contradicting the fact that $x\in \Phi$  (since $B(x, x)=1$ for any $x\in \Phi$).

\noindent\rm (ii):\quad Set $\alpha= y$, and $\beta=r_y x$. Observe that 
$$\alpha, \beta\in \Phi^+, \text{ and } B(\alpha, \beta)=-B(x, y)=-k\leq -1.$$ 
Furthermore, calculations similar to those at the beginning of this section yield that
$$L(\widehat{\alpha}, \widehat{\beta}) \cap\widehat{Q}=L(\widehat{x}, \widehat{y}) \cap\widehat{Q}=
\R((k\pm \sqrt{k^2-1})\alpha +\beta)\cap V_1.$$
Let $$\alpha_{\infty}=\R((k+\sqrt{k^2-1})\alpha+\beta) \cap V_1,$$
and
$$\alpha_{\infty}=\R((k-\sqrt{k^2-1})\alpha+\beta) \cap V_1.$$
If $\alpha_{\infty}\neq \beta_{\infty}$ (when $k\neq 1$), then it is readily seen that 
$\alpha_{\infty}$ separates $\widehat{\alpha}$ and $\beta_{\infty}$ on the line $L(\widehat{\alpha}, \widehat{\beta})$, and $\beta_{\infty}$  separates
$\widehat{\beta}$ and $\alpha_{\infty}$.
Then similar calculations as before yield that 
$B(\widehat{\alpha}, \alpha_{\infty})\geq 0$ and $B(\widehat{\alpha}, \beta_{\infty})\leq 0$.
Thus $B(\widehat{\alpha}, z)\geq 0$ for all $z$ in the closed segment ending in $\widehat{\alpha}$ and
$\alpha_{\infty}$; and $B(\widehat{\alpha}, z')\leq 0$ for all $z'$ in the closed segment ending in $\beta_{\infty}$ and
$\widehat{\beta}$. Since $x\in \Phi$, it follows from Part~(i) above that $\widehat{x}$ does not lie in the closed segment 
ending in $\alpha_{\infty}$ and $\beta_{\infty}$. Consequently, $\widehat{x}$ must be within the closed segment $\widehat{y}=\widehat{\alpha}$ and $\alpha_{\infty}$, and the desired result follows.

\end{proof}

\begin{thm}\label{lemma:intersection}
$$\bigcap_{i=0}^\infty N_i=\{\ainf\}.$$

\end{thm}

\begin{proof}
First, note that since $\ainf\in N_i$ for all $i$,  the intersection contains $\ainf$. 
Suppose, for a contradiction, that there exists $\yinf\in N_i$ for all $i$ with
$\yinf\neq\ainf$. Choose a sequence $y_i\in\Phi^+$  with $\hat{y_i}\to\yinf$ and $y_i\in A_i$.
In particular, we may assume that $y_i\notin \R a \oplus \R b$ for each $i\in \N$. 

Note that we must have $B(\ainf,\yinf)=0$. To see this note that 
for $i\ge j$ we have
$y_i\dom a_j$ and
therefore
$B(\hat{a}_j,\hat{y}_i)> 0$ and
\begin{align*}
B(\ainf,\yinf)=\lim_{j\to\infty}\lim_{i\to\infty}B(\hat{a}_j,\hat{y}_i)\ge 0.
\end{align*}
But since $\ainf,\yinf\in E$, it follows from Lemma \ref{lem:-ve} that $B(\ainf,\yinf)\le 0$, 
and consequently  $B(\ainf,\yinf)=0$. 

It follows that $L(\ainf,\yinf)$, the line containing $\ainf$ and $\yinf$, is a subset of $\Qhat$.

For each $j\in \N\cup\{\infty\}$, denote by $P_j$ the affine 2-plane in $V$ containing $\ainf, \binf, \hat{y_j}$.
Then the intersection $P_j\cap\Qhat$ is the zero set (isotropic set) of the bilinear form $B(\,,\,)$ on an affine 2-plane and therefore
must be a conic section. Since $P_\infty \cap \Qhat$ contains both the line $L(\ainf, \yinf)$ and the point $\binf$; moreover, 
since $\binf \notin L(\ainf, \yinf)$, it follows that $P_{\infty}\cap \Qhat$ consists of two distinct lines. We have two possibilities 
to consider.

(1)\quad If the two lines in $P_{\infty}\cap \Qhat$ were parallel.
By looking at $\{(r_br_a)^i\cdot\yinf \mid i\in\N\}$ one gets a contradiction to $E$ being compact. 
This is illustrated in Figure \ref{figure:1}.
\begin{figure}[htb]
\begin{tikzpicture}

\draw [thick] (0,1)--(10,1);
\draw [thick] (0,-1)--(10,-1);
\draw [thick, dashed] (-1,1)--(0,1);
\draw [thick, dashed] (-1,-1)--(0,-1);
\draw [thick, dashed] (10,1)--(11,1);
\draw [thick, dashed] (10,-1)--(11,-1);

\draw [gray] (1,-3)--(1,3); 
\node [above left] at (1,1) {$\ainf$};
\draw[fill](1,1)circle[radius=0.1];
\node [below left] at (1,-1) {$\binf$};
\draw[fill](1,-1)circle[radius=0.1];

\draw [gray] (1,3)--(5,-1); 
\node [left] at (1,3) {$\hat a$};
\draw[gray,fill=gray](1,3)circle[radius=0.1];

\draw [gray] (1,-3)--(9,1); 
\node [left] at (1,-3) {$\hat b$};
\draw[gray,fill=gray](1,-3)circle[radius=0.1];

\node [above right] at (3,1) {$\yinf $};
\draw[fill](3,1)circle[radius=0.1];

\node [below right] at (5,-1) {$r_a\cdot \yinf $};
\draw[fill](5,-1)circle[radius=0.1];

\node [above right] at (9,1) {$(r_br_a)\cdot \yinf $};
\draw[fill](9,1)circle[radius=0.1];

\end{tikzpicture}
\caption{\label{figure:1}
A configuration of $\Qhat\cap P_\infty$ that would contradict the compactness of $E$.
If $\Qhat\cap P_\infty$ consisted of two parallel lines, the sequence of points
$(r_br_a)^i\cdot\yinf\in E$ would not converge.
}
\end{figure}
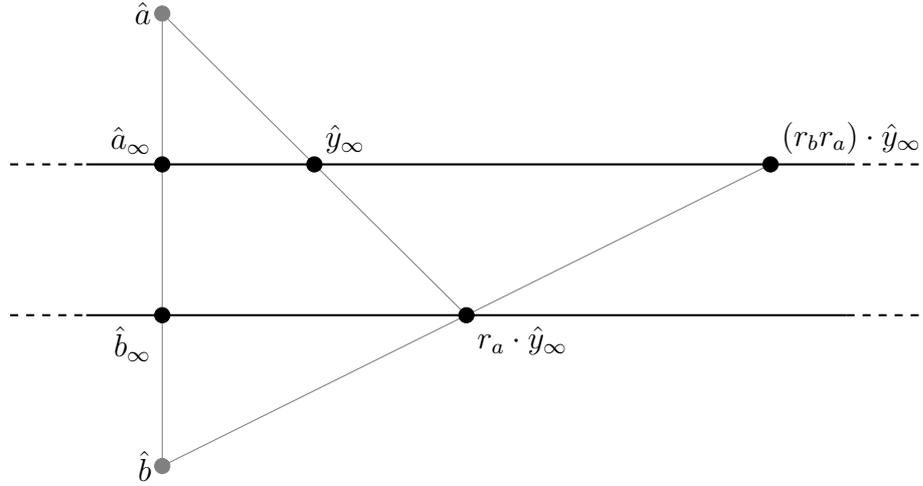

(2)\quad The two lines in $P_{\infty}\cap \Qhat$ intersect. Note first that if $\yinf$ is not the point of intersection 
of these two lines, then $B(\yinf, \binf)\neq 0$. This then implies that $B(\yinf, a)$ and $B(\yinf, b)$ are not both zero.
Then as we see in Figure~\ref{figure:2}, Figure~\ref{figure:3}  and Figure~\ref{figure:4} below that the sequence of points $(r_b r_a)^i\cdot \yinf$ will not converge to either 
$\ainf$ or $\binf$ (in fact, the sequence will converge to  the point \textcolor{blue}{$x$} as  illustrated in Figure~\ref{figure:2}), contradicting Proposition~\ref{lem:periodic}. 
\begin{figure}
\resizebox{7cm}{!}{%
\begin{tikzpicture}
        \draw (0,6)--node[pos=0.09,left]{$\hat{a}$}node [pos=0.93,left]{$\hat{b}$}(0,-6);
         \draw [color=red](-3,1) --node [color=black,pos=0.22,below]{$\hat{a}_{\infty}$}node [color=black,pos=0.39,below]{$\hat{y}_{\infty}$}node [color=black,pos=0.50,above]{$(r_{b}r_{a})\cdot\hat{y}_{\infty}$}(9,-1);

        \draw [color=red](-3,-3) --node [color=black,pos=0.22,above]{$\hat{b}_{\infty}$}node [color=black,pos=0.395,above]{$r_{a}\cdot\hat{y}_{\infty}$}(9,1);    
        \draw [color=blue](0,5) -- (2.4,-1.2);
        \draw [color=blue](0,5) -- (3.6,-0.8);
        \draw [color=blue](3.11538, -0.0192308) -- (0,-5.16132);
        \draw [color=blue](4.1, -0.168) --node[color=black,pos=0,anchor=south west]{$(r_{b}r_{a})^{2}\cdot\hat{y}_{\infty}$} node[color=black,pos=0.15,anchor=north west]{$(r_{a}r_{b}r_{a})\cdot\hat{y}_{\infty}$}(0,-5.16132);
        \draw [dotted,color=blue](4.3,-0.333)--node[pos=1,below,color=blue]{$x$}(5,-0.3333);
\end{tikzpicture}
}
\caption{\label{figure:2}
}
\end{figure}

\begin{figure}[htb]
\resizebox{8cm}{!}{%
\begin{tikzpicture}
\tikzset{%
dotstyle/.style={radius=0.05,fill=black} %
}

\coordinate (A) at (1,3.5);
\coordinate (B) at (1,-3.5);
\coordinate (AI) at (1,0.9);
\coordinate (BI) at (1,-0.9);

\draw (1,-5) -- (1,5); 

\def\Qangle{15}
\draw [name path= AIline, color=red] (AI) -- +(-\Qangle:10)  (AI) -- +(180-\Qangle:1);
\draw [name path= BIline, color=red] (BI) -- +(\Qangle:10) (BI) -- +(\Qangle-180:1);

\path [name path = Alinetop] (A) -- +(-29:15);
\draw [name intersections={of= Alinetop and AIline, by=yi}] (A) -- (yi);
\draw [name path = Blower, name intersections={of= Alinetop and BIline, by=ayi}] (B) -- (ayi);
\draw [name path= Amiddle, name intersections={of=Blower and AIline, by =bayi}] (A) -- (bayi);
\draw [name path= Bupper,name intersections={of=Amiddle and BIline, by = abayi}] (B) -- (abayi);
\draw [name intersections={of=Bupper and AIline, by=babayi}] (A) -- (babayi);

\def\dotsize{0.05}
\draw [fill=black] (yi) circle [radius=\dotsize];
\node [below] at (yi) {$\yinf$};
\draw [dotstyle] (ayi) circle;
\node [above right] at (ayi) {\small $r_a\!\cdot\!\yinf$};
\draw [dotstyle] (bayi) circle;
\node [below right] at (bayi) {\small $r_br_a\!\cdot\!\yinf$}; 
\draw [dotstyle] (abayi) circle;
\draw [dotstyle] (babayi) circle;

\node [below left] at (AI) {$\ainf$};
\draw[gray,fill=gray] (AI) circle[radius=0.05];
\node [above left] at (BI) {$\binf$};
\draw[gray,fill=gray] (BI) circle[radius=0.05];
\node [left] at (A) {$\ahat$};
\draw[gray,fill=gray] (A) circle[radius=0.05];
\node [left] at (B) {$\bhat$};
\draw[gray,fill=gray] (B) circle[radius=0.05];

\end{tikzpicture}
} 
\caption{\label{figure:3}
}
\end{figure}


\begin{figure}[htb]
\resizebox{8cm}{!}{%
\begin{tikzpicture}
\tikzset{%
dotstyle/.style={radius=0.05,fill=black} %
}

\coordinate (A) at (1,3.5);
\coordinate (B) at (1,-3.5);
\coordinate (AI) at (1,0.9);
\coordinate (BI) at (1,-0.9);

\draw (1,-5) -- (1,5); 

\def\Qangle{15}
\draw [name path= AIline, color=red] (AI) -- +(-\Qangle:5)  (AI) -- +(180-\Qangle:7);
\draw [name path= BIline, color=red] (BI) -- +(\Qangle:5) (BI) -- +(\Qangle-180:7);

\draw [name path = Blower] (B) -- +(180-35:9);
\draw [name path=Aupper,name intersections={of= Blower and BIline, by=abayi}] (A) -- (abayi);
\draw [name path = Bupper, name intersections={of = Aupper and AIline, by=bayi}] (B) -- (bayi);
\draw [name path= Alower, name intersections={of=Bupper and BIline, by =ayi}] (A) -- (ayi);
\path [name intersections={of = Alower and AIline, by=yi}]; 

\draw [dotstyle] (abayi) circle;
\draw [dotstyle]  (bayi) circle;
\draw [dotstyle]  (ayi) circle;
\draw [dotstyle]  (yi) circle;
\node [above] at (yi) {\small $\yinf$};
\node [below] at (abayi) {\small $(r_ar_br_a)\!\cdot\!\yinf$}; 


\node [above right] at (AI) {$\ainf$};
\draw[gray,fill=gray] (AI) circle[radius=0.05];
\node [below right] at (BI) {$\binf$};
\draw[gray,fill=gray] (BI) circle[radius=0.05];
\node [right] at (A) {$\ahat$};
\draw[gray,fill=gray] (A) circle[radius=0.05];
\node [right] at (B) {$\bhat$};
\draw[gray,fill=gray] (B) circle[radius=0.05];

\end{tikzpicture}
} 
\caption{\label{figure:4}
}
\end{figure}


Thus $\yinf$ must be the point of intersection of the two lines making up $P_{\infty}\cap \Qhat$. Consequently, for $j$ sufficiently
large, $P_j\cap\Qhat$ must be either a pair of intersecting lines or two hyperbole. 

Case (2a)\quad For $j$ sufficiently large, $P_j\cap \Qhat$ consists of two intersecting lines as in Figure\ref{figure:5} below:


\begin{figure}[htb]
\resizebox{8cm}{!}{%
\begin{tikzpicture}
\tikzset{%
dotstyle/.style={radius=0.05,fill=black} %
}

\coordinate (A) at (1,3.5);
\coordinate (B) at (1,-3.5);
\coordinate (AI) at (1,0.9);
\coordinate (BI) at (1,-0.9);

\def\Qangle{15}
\draw [name path= AIline, color=red] (AI) -- +(-\Qangle:8)  (AI) -- +(180-\Qangle:1);
\draw [name path= BIline, color=red] (BI) -- +(\Qangle:8) (BI) -- +(\Qangle-180:1);
\path [name intersections={of=AIline and BIline, by=P}];


\path [name path=BIlineExt, color=green] (BI) -- +(\Qangle:9);
\path [name path=AIlineExt, color=blue] (AI) -- +(180-\Qangle:1.5);
\path [name path=rightBorder, color=green] (8.7,5) -- ++(0,-5);
\path [name path=topBorder, color=blue] (-1,4) -- (9.5,4);
\path [name path=leftBorder, color=green] (0,0) -- (0,5);
\path [name intersections={of=topBorder and rightBorder, by= TR}];
\path [name intersections={of=topBorder and leftBorder, by= TL}];
\path [name intersections={of=leftBorder and AIlineExt, by= BL}];
\path [name intersections={of=rightBorder and BIlineExt, by= BR}];


\path [pattern=north east lines, pattern color=gray] (P) -- (BR) -- (TR) -- (TL) -- (BL) -- (P);

\draw (1,-5) -- (1,5); 

\node [below left] at (AI) {$\ainf$};
\draw[gray,fill=gray] (AI) circle[radius=0.05];
\node [above left] at (BI) {$\binf$};
\draw[gray,fill=gray] (BI) circle[radius=0.05];
\node [left, fill=white, rounded corners, inner sep=2pt] at ($ (A) + (-0.1,0) $) {$\ahat$};
\draw[gray,fill=gray] (A) circle[radius=0.05];
\node [left] at (B) {$\bhat$};
\draw[gray,fill=gray] (B) circle[radius=0.05];

\end{tikzpicture}
}
\caption{\label{figure:5}
}
\end{figure}


Then $\widehat{y_j}$ must be in the shaded region. Otherwise, consider the line through $\widehat{a_j}$ and
$\widehat{y_j}$. On this line there is at least one point of $\Qhat$ in between $\widehat{a_j}$ and $\widehat{y_j}$, 
contradicting the assumption that $y_j\dom a_j$ by Proposition~\ref{pp:sep} above. But if $\widehat{y_j}$ is in the shaded region, by choosing $i$ large
enough, we see that on the line through $\widehat{a_i}$ and $\widehat{y_j}$ either the two normalized roots $\widehat{a_i}$
and $\widehat{y_j}$ are in between the two points of intersection of that line with $\Qhat$, or else $\widehat{a_i}$ separates 
$\widehat{y_j}$ and the two points of intersection of that line with $\Qhat$. By Proposition~\ref{pp:sep}, the former is impossible, and the latter implies 
that $a_i\dom y_j\dom a_j$, and Lemma~\ref{lemma:interlaced} then establishes that $\yinf=\ainf$, a contradiction to our assumption.

Case (2b)\quad For $j$ sufficiently large, $P_j\cap \Qhat$ consists of two hyperbole as in Figure~\ref{figure:6}


\begin{figure}[htb]
\resizebox{8cm}{!}{%
\begin{tikzpicture}
\tikzset{%
dotstyle/.style={radius=0.05,fill=black} %
}

\coordinate (A) at (1,3.5);
\coordinate (B) at (1,-3.5);
\coordinate (AI) at (1,0.9);
\coordinate (BI) at (1,-0.9);

\def\Qangle{15}
\path [name path= AIline, color=red] (AI) -- +(-\Qangle:8)  (AI) -- +(180-\Qangle:1);
\path [name path= BIline, color=red] (BI) -- +(\Qangle:8) (BI) -- +(\Qangle-180:1);
\path [name intersections={of=AIline and BIline, by=P}];

\draw [name path=vert] (1,-5) -- (1,5); 

\draw [name path = upperCurve, color=red] ($ (P) + (-4,1.5)  $) .. controls (P) and (P) .. ($ (P) + (4,1.5)  $);
\draw [name path = lowerCurve, color=red] ($ (P) + (-4,-1.5)  $) .. controls (P) and (P) .. ($ (P) + (4,-1.5)  $);

\path [name intersections = {of = upperCurve and vert, by=AIb}];
\path [name intersections = {of = lowerCurve and vert, by=BIb}];

\path [name path = upperRegion, pattern=north east lines, pattern color=gray] ($ (P) + (-4,1.5)  $) .. controls (P) and (P) .. ($ (P) + (4,1.5)  $) 
                                                              -- ++(0,2) -- ($ (P) + (-4,1.5) + (0,2)$) -- cycle;

\node [below left] at (AIb) {$\ainf$};
\draw[gray,fill=gray] (AIb) circle[radius=0.05];
\node [above left] at (BIb) {$\binf$};
\draw[gray,fill=gray] (BIb) circle[radius=0.05];

\end{tikzpicture}
} 
\caption{\label{figure:6}
}
\end{figure}


Then similar as in Case (2a), we get a contradiction as $\widehat{y_j}$ can not be in any region of this affine 2-plane.

\end{proof}

\begin{cor}  \label{cor:fundsys}
The set $\{N_i\| i\in \N  \}$ forms a fundamental system of neighbourhoods of the point $\ainf\in E$.
\end{cor}

\begin{proof}
That each $N_i$ is a neighbourhood of $\ainf$ is the statement of Lemma~\ref{lemma:neighbourhood}.
That each neighbourhood of $\ainf$ contains an $N_i$ follows from Theorem~\ref{lemma:intersection}. 
\end{proof}

The limit set $E$ is equipped with the induced topology from $V$.
Let $ W''\le W$ be any rank 2 reflection subgroup of $W$. 
Note that for a fixed Coxeter system any reflection subgroup has a unique Coxeter generating set.
There is a unique canonical generating set for $W''$. If $W''$ is infinite and non-affine, then the above argument provides a system of neighbourhoods around each of its limit points. Let $\cN$ denote the collection of all such sets $\Int N_i$ for all possible $i$ and $W''$. That is,
$$
\cN=\{\Int(N_i(\hat x)) \| i\in\N, \hat{x}\in E_2  \}
$$
where $E_2$ denotes the subset of $E$ consisting of all limit points of non-affine infinite dihedral reflection subgroups. It is shown in \cite[Theorem 3.1]{DHR13} that the the set of limit points coming from non-affine infinite dihedral reflection subgroups of $W$ is dense in $E$. Indeed, they show that $W\cdot\ainf$ is dense in $E$. We hope to prove that the set $\cN$ is a basis for the topology on $E$, but this remains to be a work in progress.


%


\end{document}